\documentclass{article}
\usepackage{blindtext}
\usepackage{titlesec}

\usepackage{hyperref}
\hypersetup{
    colorlinks=true,
    linkcolor=blue,
    filecolor=magenta,
    citecolor=blue,
    urlcolor=blue,
    pdftitle={Overleaf Example},
    pdfpagemode=FullScreen,
    }
\usepackage{arxiv}

\usepackage[utf8]{inputenc} 
\usepackage[T1]{fontenc}    
\usepackage{hyperref}       
\usepackage{url}            
\usepackage{booktabs}       
\usepackage{amsfonts}       
\usepackage{nicefrac}       
\usepackage{microtype}      
\usepackage{lipsum}
\usepackage{amsmath}
\usepackage{breqn}

\usepackage[mathscr]{eucal}

\usepackage{amssymb}
\usepackage{amsthm}
\usepackage{amsmath}
\usepackage{breqn}

\usepackage[all]{xy}
\usepackage{relsize}

\usepackage{csquotes}

\usepackage{lineno}
\usepackage{graphicx}
\usepackage{epsfig}
\usepackage{amsthm}
\usepackage{amsmath}
\usepackage{latexsym}
\usepackage{amsfonts}
\usepackage{amssymb}
\usepackage[all]{xy}
\usepackage{amsmath,amsfonts,amssymb,amscd,amsthm,xspace}
\usepackage[utf8]{inputenc}
\usepackage{relsize}
\newtheorem{theorem}{Theorem}

\usepackage[normalem]{ulem}

\usepackage{natbib}

\title{A Differential Algebraic Framework for Boundary Layer $\alpha$-Models}

  \author{
  C. Valencia-Negrete, 
J. Flores, 
M. Perez, 
M. Romero de Terreros,
M. Favela \\
Department of Physics and Mathematics\\
Universidad Iberoamericana Ciudad de M\'{e}xico\\
  Prolongaci\'{o}n Paseo de la Reforma 880,  Mexico City ---01219, MEXICO \\
  \texttt{carla.valencia@ibero.mx} 
}

\begin{document}
\maketitle

\begin{abstract}
By addressing the problem of describing the differences between the saturation monoids of the Prandtl and Cheskidov boundary layers, and algebraically distinguishing the action of the Helmholtz operator, we can show that, when the modification of the vector field is halted at a minimum scale determined by the differential filter, the paradox of constructing a solution that reaches infinite velocity in a finite time is avoided. Furthermore, it allows us to obtain the coefficients of the Taylor expansion near the wall for boundary layers where the Helmholtz operator has been applied.
\end{abstract}

\keywords{Turbulence theory, Vortex dynamics }

\tableofcontents

\section{Introduction} 

The elements of a system disappear if they remain isolated. 
This gives rise to the idea of presenting algebra 
as the study of relationships between similar objects 
through the construction of a graphical representation of the product, 
where prime numbers become the main nodes of a network 
that generates the natural numbers, 
and the integers become the complete face of a new world 
in which divisors have a designed and identifiable location.
In addition, any other multiplicative group can be represented in the same way, 
enabling a direct analogy with polynomials with complex coefficients $\mathbb{C}[x]$, 
as generated by prime polynomial combinations.
This is 
how we understand the way \emph{combinatorial commutative algebra} visualizes 
monomials as \emph{geometric objects} \cite[p.~61]{Miller2005} 
and \emph{monomial ideals} by \emph{staircase diagrams} \cite[p.~47]{Miller2005}
for systems of polynomial equations.

The structure of the boundary layer equations, as sums of terms that are products of the velocity and pressure derivatives, originates from Prandtl's formulation \cite{prandtl1904,Prandtl1928}. 
Each monomial can be viewed as a word, 
with each derivative being a letter of the alphabet used to write it. 
Thus, these equations are differential polynomials \cite{ritt1950,kolchin1973}. 
As a result, we may \emph{see} the space occupied by the multiplicative structure, 
which gives rise to the idea of identifying structural differences 
between the multiplicative polynomials that form the Prandtl limit layers and 
the ``\emph{Prandtl-like boundary-layer approximation of the Leray-$\alpha$ model}''
of Cheskidov, Holm, Olson and Titi \cite{cheskidov2005}. 
Unlike the former, this $\alpha$-boundary layer model allows the existence, 
uniqueness and computation of exact solutions. 

In differential algebra, the corresponding multiplicative structure  
is called the \emph{saturation monoid}, denoted as $H^\infty$. 
This monoid is generated by the finite products of all the \emph{initials}, 
or leading coefficients,
and \emph{separants},
derivatives of the leading coefficients, of a differential system. 
We give a brief introduction to these concepts in Section $\mathbf{1}$.
The difference between the classical and the $\alpha$ boundary layer  models 
is the inclusion of the Helmholtz operator
$\mathbf{u}^\alpha - \alpha^2 \Delta \mathbf{u}^\alpha = \mathbf{u}$. 
This is a specific type of \emph{fluctuation}.
In Reynolds-averaged Navier-Stokes (RANS) formulations, 
fluid velocity is separated into a linear combination of 
mean flow $\mathbf{\bar{u}}$ and turbulent fluctuation $\mathbf{u'}$,
$\mathbf{u} = \mathbf{\bar{u}} + \mathbf{u'}$ \cite{majda2016introduction,pope2000turbulent}.
Under the \emph{Reynolds decomposition}, the fluctuation $\mathbf{u'}$ is treated as 
an unknown random variable to be modeled statistically \cite{berselli2021three}.
However, applying the Helmholtz $\alpha$-filter explicitly defines this spatial fluctuation as $\mathbf{u'}=-\alpha^2 \Delta \mathbf{u}^\alpha$ .

This acts as a \emph{filter}, more specifically
a \emph{differential filter} \cite[p.~233]{berselli2021three},
taking a twice continuously differentiable velocity field $\mathbf{u}$ 
in a bound domain and introducing an infinitesimal fluctuation. 
Thus, the parameter $\alpha$ can be used, as we show in Section $\mathbf{5}$, 
as the width of the differential filter \cite{cheskidov2005}, associated with the resolution of the computational grid \cite{HAFEZ20071588}.
The Navier-Stokes-$\alpha$ model, or viscous Camassa-Holm equations,
was introduced as a closure approximation model for turbulent flows 
by Chen, Foias, Holm, Olson, Titi, and Wynne \cite{CHEN199949,FOIAS2001505}. 
The Navier-Stokes equations,
where the boundary conditions are usually considered to be null or periodic, 
and the boundary layer models where the fluid accelerates from  zero at the wall to the freestream velocity on the other side,
are inherently different in spite of being similar.
In $2001$, Marsden and Shkoller proved the global well-posedness for the Lagrangian
averaged Navier–Stokes (LANS-$\alpha$) equations on bounded domains \cite{marsden2001global}.
By incorporating this operator into fluid models, 
the Leray-$\alpha$ model given by Cheskidov, Holm, Olson and Titi \cite{cheskidov2005} also satisfy existence and uniqueness qualities.

From the perspective of differential algebra,  
we ask what the difference between the saturation monoid $S$ 
corresponding to the Prandtl limit layer 
and the saturation monoid $S_\alpha$ 
corresponding to the Leray-$\alpha$ model is
that allows the existence and uniqueness of analytic solutions. 
As a result, we show in 
Theorem $\mathbf{1}$ of Section $\mathbf{2}$
that the boundary layer $\alpha$-models
satisfy the algebraic properties established in 
the Rosenfeld Lemma.
In consequence, each finite differential polynomial system
given by Goldstein's method near the wall
has a triangular form which allows us to clearly  and recursively compute
the coefficients of the Taylor approximations of its solutions,
up to the order needed. 

Moreover, the recent construction of a solution that reaches infinite velocity 
in finite time in the Navier–Stokes equations \cite{openai2026navierstokes}, 
built on Buckmaster, Vicol, Alp\"{o}ge, Coiculescu, C\'{o}rdoba, Mart\'{i}nez-Zoroa 
ideas and others on singularity formation \cite{córdoba2025finitetimesingularitiessmooth,shkoller2026incompressibleeulerblowupc1frac13},
coupled with the fact that the Helmholtz filter 
reflects the physical condition that there is a minimum scale 
inherent in the macroscopic representation rather than transmitting the point pulse, 
motivated us to modify our initial text to  that,  
from an algebraic point of view, 
\sout{that} the Cheskidov boundary layer does not generate this paradox 
or allow for the accumulation of rotational energy in this manner. 
Demonstrating that the separant $S_\alpha$ becomes a strictly positive 
structural constant, proportional to $\alpha^2 \nu$, 
shows that setting this separant to zero  
yields an empty set, which
formally proves  by reduction to the absurd that the singularity cannot exist.

\section{Differential Algebra Concepts}

One way to find solutions to a system of differential equations is to conjugate them.
This involves combining the given conditions in each equation with those in the others,
thereby creating equivalent expressions for the original system 
that are easier to study. 
A clear example of this is the triangulation of systems of 
linear ordinary differential equations. 
In this case, the process involves products and linear combinations 
of the original equations. 
On the other hand, the boundary layer models are algebraic but non-linear. 
However, the possibility of obtaining an equivalent and triangular
expression of the system is translated into the language of 
\emph{differential algebra theory} as the problem of belonging to
a \emph{differential ideal}, and answered by the \emph{Rosenfeld Lemma}
\cite{rosenfeld1959} in terms of the \emph{saturation monoid} structures. 

Let's recall some basic definitions about rings.
(See \cite{lang_algebra_2005,dummit2004abstract,Miller2005,cox2015ideals}).
Let $R$ be a set with a \emph{ring} structure. 
This is, there are two binary operations well defined in $R$, 
typically called addition ($+$) and product ($\cdot$), 
that satisfy commutativity and associativity under addition, 
and associativity of the product.
It contains an additive identity (usually denoted as $0$)
such that every element has an additive inverse,
and the product distributes over addition from both the left and the right.
More specifically, 
a \emph{polynomial ring}, typically denoted as $R[x]$, 
is the set of all formal polynomials in a variable $x$, 
with coefficients $a_k\in R$:
$p(x)=a_0+a_1 x+ \cdot \cdot \cdot +a_n x^n$,
$k\in \{0,...,n\}$.
An \emph{ideal} $I$ is a subset of a ring where, 
if you take the product of an element in the ideal by any element in the ring, 
the result stays in the ideal, 
and is closed under addition.
In particular, a \emph{monomial ideal} is a specific type of ideal in a polynomial ring 
that is generated entirely by single terms with variables 
and non-negative integer exponents.
In general, a generator of the ideal is a polynomial,
which are simpler because they are just products of the variables.
In a monomial ideal, membership is dictated exclusively by divisibility: 
a polynomial is in the ideal if and only if every term within 
it is divisible by one monomial generators.
Dickson's Lemma guarantees that every monomial ideal in a polynomial ring with finitely many variables is finitely generated.
In order to distinguish the pertinence or priority of the monomials, 
we may define a \emph{partial order},
a \emph{ranking} between the variables, denoted by $'\prec \ '$.
Once the ranking has been chosen,
the \emph{Leading Monomial} ($LM$) is the specific combination of variables 
and exponents that ranks highest under the chosen monomial order,
the \emph{Leading Coefficient} ($LC$) is the scalar element
attached to the leading monomial, 
and the \emph{Leading Term} ($LT$) is the product 
of the leading coefficient and the leading monomial.

Below, we give some differential algebra concepts.
See \cite{ritt1950,kolchin1973,kaplansky1957introduction,boulier2019differential,cox2015ideals,fakouri2018new,harrington2016reduction,hubert2000factorization,10.1007/3-540-51082-6_73}.
A \emph{differential ring}, $\mathfrak{R}\{x\}$, 
is a polynomial ring equipped with a 
derivation operator, 
a mapping $\delta: \mathfrak{R}\{x\} \to \mathfrak{R}\{x\}$,
satisfying the Leibniz product rule $\delta(ab) = \delta(a) \ b + a \ \delta(b)$.
The differential polynomials are made with products and additions of 
\emph{differential indeterminates}.
If instead of only one, we consider a set of $\mathbf{y}=\{y_1,...,y_n\}$ variables.
Then, the differential polynomial ring, $\mathfrak{R}\{\mathbf{y}\}=\mathfrak{F}\{y_1,...,y_n\}$ has differential polynomials with terms that contain 
$y_i$ powers, for $i \in \{1,...,n\}$,
and any of its formal derivatives of any order.
A \emph{differential ideal} is an ideal inside a differential ring that 
is closed under differentiation.
Formally speaking, $I$ is a differential ideal if 
$a \in I$ implies $\delta(a) \in I$.
This way, the differential ideal is the set where we have to look for
the simpler expressions of the problem, 
because it contains all the possible differential
polynomials equivalent to the original through algebraic transformations
and derivations. 
When looking for a \emph{differential monomial ideal}, it is best to select monomials generated by terms that provide information about the behavior represented by the system. 
Choose those with the highest order of differentiation in the variable that determines the essential behavior of the layer's deformation, such as \emph{shear stress}.
The highest-ranking derivative, of the prioritized variable, 
present in a differential polynomial $P$ is called 
its \emph{leader}, denoted $u_P$. 
If we rewrite $P$, it takes the form $P = I_{P} (u_P)^d + \dots + P_0$. 
Unlike the standard leading coefficient ($LC$), which is just a scalar number, 
the initial $I_P$ is usually a polynomial itself, comprised of lower-ranking derivatives.
The coefficient $I_P$ is called the \emph{initial} of $P$.
A \emph{differential monomial ideal}, 
or \emph{differential initial ideal}, denoted by  $\text{in}(I)$, 
is the ideal generated by this infinite set of leading monomials.

Now, we take one step into the systems of differential polynomials.
See \cite{ritt1950,kolchin1973,boulier1995,boulier2009computing,hubert2000factorization,harrington2016reduction,fakouri2018new}.
Let $p_i \in \mathfrak{R}\{x\}$, for $i\in\{1,...,m\}$, be a differential polynomial.
Then, there is system of algebraic differential equations represented as 
$\Sigma = \{P_1, P_2, \dots, P_k\}$.
A \emph{differential ideal generated by} $\Sigma$, denoted by $[\Sigma]$, 
is formed by taking the original polynomials $P_i$, 
multiplying them by any other differential polynomials in $\Sigma$, 
adding them together, and taking formal derivatives of them to any order.
So that, the \emph{easier} expression of the system $\Sigma$,
if exists, is made by elements of the differential ideal generated by $\Sigma$.
An \emph{autoreduced set} is a finite system of differential equations that has been simplified so that no equation can be used to eliminate terms in any other equation in the set.
The problem of deciding if a differential polynomial belongs 
to an ideal has the difficulty that the elements of the differential ring
are generated by a combination of an infinite set of variables and their derivatives
because a differential ideal 
might require an infinite sequence of generators.
One way to allow the reduction of the complexity of the problem
is to truncate the system. 
Boulier, Lazard, Ollivier, and Petitot assumed 
the existence of an analytical solution expressed as a
Taylor series near a boundary condition.
This maps the continuous differential polynomials into a finite system 
of algebraic polynomials, allowing you to sequentially solve for the first $N$ coefficients of the Taylor approximation from a finite system of 
differential polynomials for the first coefficients of the approximation
that will satisfy the boundary condition. 

This is the base for the Rosenfeld-Groebner algorithm, which is specifically designed to obtain a simpler expression of the system,
and the \href{https://codeberg.org/francois.boulier/DifferentialAlgebra}{DifferentialAlgebra} Phyton package developed by Boulier and Lemaire.
To read more about the concepts and properties stated in this paragraph, see \cite{boulier1995,boulier2009computing,boulier2019differential,harrington2016reduction,hubert2000factorization}.
A polynomial belongs to a monomial ideal if and only if every single term within the polynomial is a multiple of  at least one of the monomial generators, leading or initial terms of the system.
But, since the general structure is a ring, there is no actual division, but rather a pseudo-division.
If $h$ is an initial of one element $P_k$ in the polynomial differential system, 
and $h^d$ is one of its multiples, for a fixed $d \in \mathbb{N}$,
\emph{differential pseudo-division} extends polynomial pseudo-division
to obtain $h^d P_i = QG + R$,
where the product $h^d P_i$ of the initial power $h^d$
and a differential polynomial $P_i$
in the system is expressed as a linear combination of a
lower order differential polynomial $G$ and a reminder $R$.
For this differential pseudo-division to actually reduce $P_i$,
we must have $h^d \not =0$.
\emph{Saturation monoids} are precisely those that fulfill this role within the algorithm.
Given a differential polynomial $P= I_{P} (u_P)^d + \dots + P_0$,
the \emph{separant} $S_P$ is the formal partial derivative $\frac{\partial P}{\partial u_P}$
with respect to the leader $u_P$.
A \emph{saturation monoid} of a system of differential polynomials $\Sigma$, 
denoted as $H_{\Sigma}^\infty$, is 
the monoid generated by multiples of all initials $I_P$ and separants $S_P$
of $P\in \Sigma$.
In order to perform pseudo-division, 
we must ensure that $H_{\Sigma}^\infty \neq 0$,
so that any $S_P \neq 0$, for $P\in \Sigma$.
This general rule, translated to the boundary layer case, is the distinction
between being able to find a formal Taylor expression for an analytic
solution near the wall in the boundary layer-$\alpha$ models,
and having an autoreduced system in the original boundary layer models.

\section{Triangular systems for Leray-\texorpdfstring{$\alpha$}{alpha} boundary layers}

Thus, reducibility via pseudo-division is determined by the structure 
of the saturation monoid, the representation of which contains the product lattice 
of the initial terms of the polynomial. 
In the study of the boundary layer-$\alpha$ models, we consider two elements. 
Firstly, the Taylor series expansion of the solutions to the Prandtl model near the wall. This was proposed by Goldstein \cite{goldstein1948}
by assuming the velocity components can be expanded in series powers of $y$, 
with coefficients as functions of $x$.
Then, substitute in the boundary-layer equations and equate powers of $y$
to study the generation of the singularity that causes fluid separation when it moves horizontally. 
Secondly, the difference between the lattice structures of each saturation monoid in the original case and those of the boundary layer-$\alpha$ models must be considered.
Applying the Helmholtz operator to Prandtl's boundary layer equations 
yields the Leray-$\alpha$ \cite{cheskidov2005}.
We show that this operator raises the spatial order of the normal derivatives,
from 2nd to 4th order, averting the algebraic breakdown known 
as the Goldstein singularity. 
By choking the high-frequency cascade necessary to drive the stress corrections 
in the blow-up proof, the $\alpha$-filter actively prevents the finite-time blow-up
process from taking root.

The spatial coordinate measuring the distance 
downstream along the boundary wall is 
$x \in (0,L) \subset \mathbb{R}$,
for a fixed length $L>0$, 
which is considered the \emph{streamwise} direction.
The wall-normal coordinate that measures the perpendicular 
distance from the solid boundary into the fluid is 
$y \in (0,\delta) \subset \mathbb{R}^+$, 
for a fixed value $0<\delta <\!<\!<1$.  
Thus, the simplest domain considered will be the
open rectangle $\cal{D}$ $ \subset \mathbb{R}^2$. 

Let $\mathfrak{R} = \mathfrak{F}\{U, V, p\}$ be 
the ring of differential polynomials in the state variables 
with derivations $\partial_x$ (streamwise) and $\partial_y$.  
The variables $(U, V)$ admit formal power series expansions 
in the wall-normal coordinate $y$ in a neighborhood 
of the boundary $y = 0$ \cite{goldstein1948}:

$$U(x,y) = \sum_{k=1}^\infty a_k(x) y^k, \quad V(x,y) = \sum_{k=1}^\infty b_k(x) y^k,$$

where  $a_0 = b_0 = 0$ is the no-slip boundary condition.  
For the incompressible case, 
$\partial_x U + \partial_y V = 0$. 
This induces the coefficient relation: 
$b_k(x) = -\frac{1}{k} \partial_x a_{k-1}(x)$.  
To evaluate the polynomial system calculations for the general $\alpha$-boundary layer, we must independently derive the coefficient of $\mathcal{O}(y^m)$.
Let the filtered velocity be defined by the Helmholtz operator $W = U - \alpha^2 \partial_{yy} U$. Expanding the velocity field as a formal Taylor series at the wall yields $U(x,y) = \sum_{k=1}^\infty a_k(x) y^k$.
Applying $\nu \partial_{yy} W $ provides:
\begin{equation}\label{eq:1}
\nu \partial_{yy} W = \nu \sum_{m=0}^\infty (m+2)(m+1) w_{m+2} y^m.
\end{equation}
Substituting the coefficient  $w_{m+2} = a_{m+2} - \alpha^2 (m+4)(m+3) a_{m+4}$ expands $\nu \partial_{yy} W $  to:
\begin{equation}\label{eq:3}
\nu \partial_{yy} W = \sum_{m=0}^\infty \left[ \nu(m+2)(m+1)a_{m+2} - \alpha^2 \nu \frac{(m+4)!}{m!} a_{m+4} \right] y^m.
\end{equation}
\vspace{4mm}

\begin{theorem}
Assume the velocity components $U,V \in \cal{C}^{\infty}(D)$
verify the conservation of momentum equation in Leray-$\alpha$ model:

\begin{equation}\label{eq:3b}
U \partial_x W + V \partial_y W = \nu \partial_{yy} W - p_x.
\end{equation}

Then, for $k \ge 4$, the hydrodynamic coefficients satisfy the recurrence relation:

\begin{equation}\label{eq:2}
a_k = \frac{1}{\alpha^2 k(k-1)} \left[ a_{k-2} - \frac{C_{k-4}(x) + p_x \delta_{k-4,0}}{\nu(k-2)(k-3)} \right].
\end{equation}
\end{theorem}

\begin{proof}
We impose a strict spatial elimination ranking ($a_k \prec a_{k+1}$) on the coefficient ring. For the generalized polynomial $P_{\alpha, m}$, the structural curvature $a_{m+4}$ becomes the leader.
The separant changes from a dynamic flow variable in the original Prandtl boundary layer
into a structural constant:

$$S_\alpha = \frac{\partial P_{\alpha, m}}{\partial a_{m+4}} = \alpha^2 \nu \frac{(m+4)!}{m!}.$$
 
Because the filter width $\alpha$ and the kinematic viscosity $\nu$ are strictly positive physical parameters, $S_\alpha \neq 0$. By Ritt's Radical Splitting Theorem, the singular branch of the differential ideal, $\{ \Sigma_\alpha \cup (\alpha^2 \nu) = 0 \}$, contains a non-zero constant and therefore evaluates to the empty set ($\emptyset$). The ideal is globally regular. Consequently, the general recurrence formula resolves explicitly and safely over the constant Saturation Monoid $M_\alpha = \langle \alpha^2 \nu \rangle^\infty$:

$$a_{m+4} = \frac{1}{\alpha^2 (m+4)(m+3)} \left[ a_{m+2} - \frac{C_m(x) + p_x \delta_{m,0}}{\nu(m+2)(m+1)} \right].$$

This is, by Eq. (\ref{eq:2}) in Eq. (\ref{eq:3}) at order $\mathcal{O}(y^m)$, we isolate the Cauchy product of the non-linear convective terms, denoted as $C_m(x)$, and the pressure forcing $p_x \delta_{m,0}$:

$$C_m(x) + p_x \delta_{m,0} = \nu(m+2)(m+1)a_{m+2} - \alpha^2 \nu \frac{(m+4)!}{m!} a_{m+4},$$
where:

$$C_m(x) 
= \sum_{i=1}^m a_i \dot{w}_{m-i} 
- \sum_{i=2}^m \frac{m-i+1}{i} \dot{a}_{i-1} w_{m-i+1},$$

for $w_i = a_i - \alpha^2(i+2)(i+1)a_{i+2}$. 
Therefore, 
the maximum index present in the convective term 
$C_m$ is $a_{m+3}$.

To construct the differential polynomial $P_{\alpha, m} = 0$, we rearrange the terms to isolate the highest-order spatial derivative:

$$P_{\alpha, m} := \alpha^2 \nu \frac{(m+4)!}{m!} a_{m+4} - \nu(m+2)(m+1) a_{m+2} + C_m(x) + p_x \delta_{m,0} = 0.$$

We choose a ranking where $a_k \prec a_{k+1}$: 
$$p \prec a_1 \prec a_2 \prec a_3 \dots 
\prec \dot{a}_1 \prec \dot{a}_2 \prec \dot{a}_3 \dots.$$
Then, the highest variable in $P_{\alpha, m}$ is $a_{m+4}$. This variable is formally designated as the leader of the polynomial. In classical fluid boundary layers, where $\alpha = 0$, the convective terms $U \partial_x U + V \partial_y U$ have $\dot{a}_1$ as the leader. However, in the $\alpha$-regularized system, the partial derivative of the corrected $P_{\alpha, m}$ with respect to its leader $a_{m+4}$ yields the a \emph{separant}:

$$S_\alpha = \frac{\partial P_{\alpha, m}}{\partial a_{m+4}} = \alpha^2 \nu \frac{(m+4)!}{m!}.$$ 
  
By Ritt's Radical Splitting Theorem \cite{boulier2019differential}, the validity of any differential algebraic triangulation requires evaluating the Singular Ideal branch: $\{ \Sigma \cup S_\alpha = 0 \}$.
Because the filter width $\alpha$ and kinematic viscosity $\nu$ are strictly positive physical parameters, $S_\alpha$ evaluates to a strictly positive structural constant. Since $S_\alpha \neq 0$, the singular ideal contains a non-zero constant, meaning its mathematical variety is the empty set ($\emptyset$). The pseudo-division required to calculate the boundary layer coefficients never encounters a division by zero.

Evaluating the general filtered momentum equation $U \partial_x W + V \partial_y W = \nu \partial_{yy} W - p_x$ at order $\mathcal{O}(y^m)$ establishes the algebraic balance:

$$C_m(x) + p_x \delta_{m,0} = \nu(m+2)(m+1)a_{m+2} - \alpha^2 \nu \frac{(m+4)!}{m!} a_{m+4}.$$

Isolating the leader $a_{m+4}$ yields the base polynomial structure:

$$a_{m+4} = \frac{1}{\alpha^2 (m+4)(m+3)} \left[ a_{m+2} - \frac{C_m(x) + p_x \delta_{m,0}}{\nu(m+2)(m+1)} \right].$$
\end{proof}
\vspace{3mm}

\section{The \texorpdfstring{$\alpha$}{alpha}-floor of the Helmholtz-filtered Navier–Stokes blow-up}

The finite-time blowup constructed in the OpenAI Navier-Stokes proof relies on a concentrating inner core where the radial and axial scales shrink while the velocity scales diverge. 
To halt the cascade, we may apply the generalized Helmholtz filter $\mathcal{H}_\alpha = (1 - \alpha^2 \partial_{yy})$ to the convective transport. This singular perturbation injects the Leray-$\alpha$ microscopic sub-layer. 
The OpenAI blowup requires the radial wavelength of the oscillatory pulses
to shrink to zero, which forces spatial derivatives to explode. However,
the generalized $\alpha$-recurrence mathematically guarantees that every
high-order spatial derivative (the Taylor coefficient $a_{m+4}$) is
strictly capped by the inverse scale $1/\alpha^2$. This constant
denominator acts as a shield, and it sets in at the
wavenumber $k = 1/\alpha$, where $\alpha^{2}k^{2} = 1$: the filter
transmits one half of a pulse of wavelength $2\pi\alpha$, $2.5\%$ of a 
pulse of wavelength $\alpha$, and the suppression grows as
$\alpha^{2}k^{2}$ beyond (Section~\ref{subsec:escudo}). Any oscillatory
pulse attempting to shrink its spatial period below that scale is
progressively choked by the Coherent Autoreduced Set. Therefore, for any
$\alpha > 0$, the frequency cascade is algebraically blocked, and the
generation of infinite velocity waves is topologically impossible.

\begin{theorem} 
For the generalized $\alpha$-regularized boundary layer, the differential ideal forms a globally regular Coherent Autoreduced Set over the Saturation Monoid $M_\alpha = \langle \alpha^2 \nu \rangle^\infty$. This algebraic structure strictly bounds all spatial derivatives by $1/\alpha^2$, creating an absolute spectral shield that topologically prevents the generation of infinite velocity waves (such as those driven by high-frequency oscillatory pulses), completely bypassing the Goldstein singularity.
\end{theorem}

\begin{proof} 
Let us assume the contrary and suppose that a solution to the $\alpha$-regularized boundary layer equations can develop an unbounded velocity in finite time.
Because the singular ideal is empty, the differential ideal is globally regular. The non-vanishing separant algebraically guarantees that every spatial derivative remains strictly bounded by a factor proportional to $1/\alpha^2$. The spatial derivatives cannot cascade to infinity, which directly contradicts our initial assumption that the velocity becomes unbounded.
\end{proof}

\subsection{The floor, illustrated on the exact blow-up}
\label{subsec:piso}

\paragraph{Scope.}
Figures~\ref{fig:piso-composite} and~\ref{fig:piso-strip} illustrate the
operator at the centre of Theorems~1 and~2; they are not numerical evidence
for them.
The Helmholtz operator $\bar u - \alpha^{2}\Delta\bar u = u$ is applied as a
\emph{static} filter, frame by frame, to the \emph{exact} self-similar
blow-up ansatz $u(x,y;\tau)$ of~\cite{openai2026navierstokes} along its
countdown $\tau = T - t$. No equation of motion is integrated, so nothing
below is the dynamics of the $\alpha$-models (Leray-$\alpha$, LANS-$\alpha$)
or of their boundary-layer approximation. The regularity statement of the
article rests on the algebraic argument --- the singular branch is empty
because the separant is the strictly positive constant
$\alpha^{2}\nu\,(m{+}4)!/m!$ --- and the figures show, physically, why the
operator that argument is built from returns a finite field when fed a
diverging one. (The filter is the isotropic, two-dimensional form of the
operator $\mathcal{H}_\alpha = 1 - \alpha^{2}\partial_{yy}$ of the surgery in
part~III, inverted: $\bar u = (1-\alpha^{2}\Delta)^{-1}u$. There it
regularizes the differential ideal; here one watches its action on the
field.) Throughout, the \emph{filtered peak} is $\max|\bar u|$ on the plane
$z=0$ at a given instant, and the \emph{floor} is its maximum along the
countdown, reached at $\tau = \tau_{\mathrm f}$.

\begin{figure}[t]
  \centering
  \includegraphics[width=\linewidth]{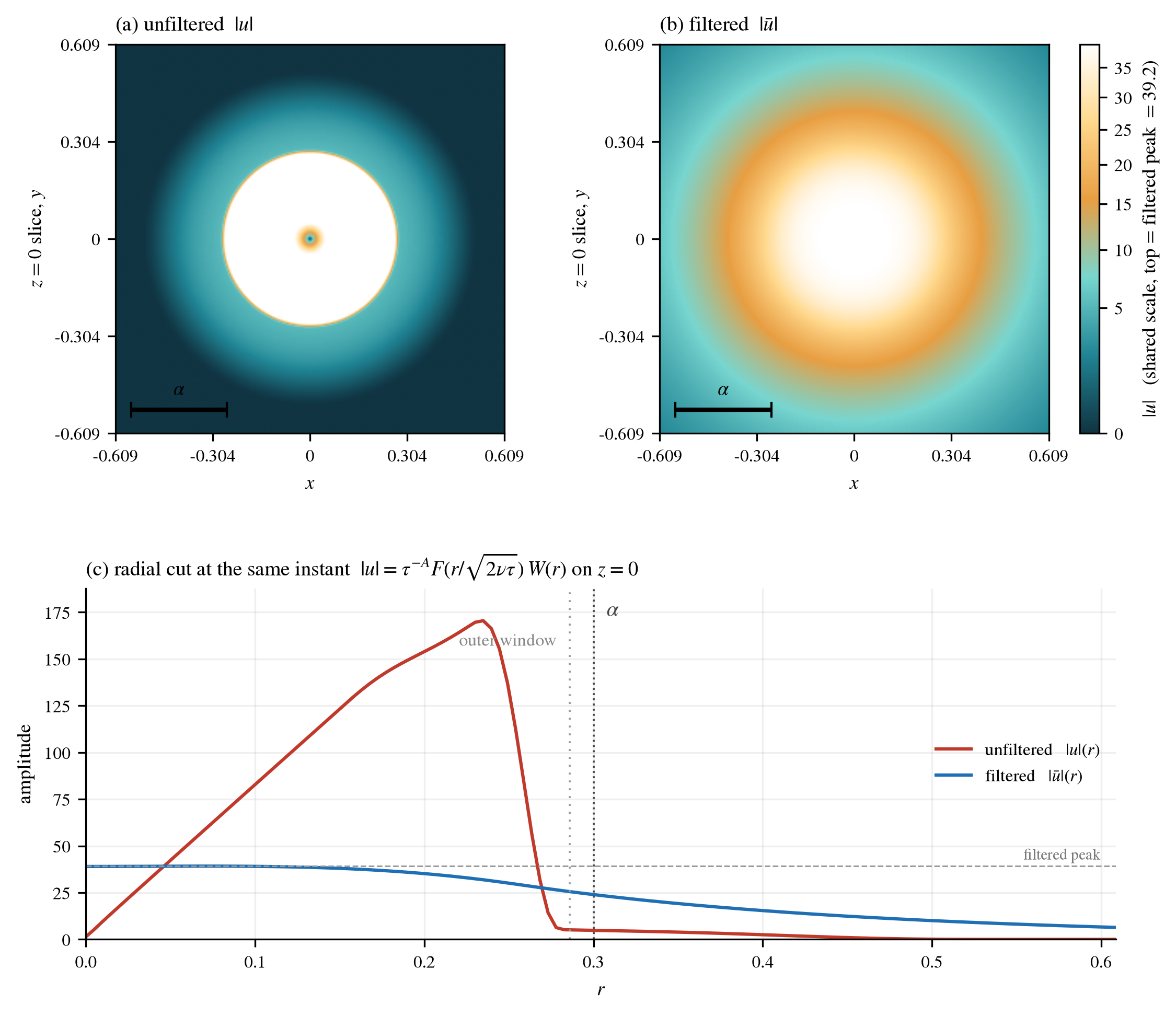}
  \caption{\textbf{The $\alpha$ floor of the filtered blow-up.}
  The self-similar blow-up ansatz of~\cite{openai2026navierstokes} on the
  mid-plane $z=0$ at $\tau = \tau_{\mathrm f} = 2.92\times10^{-3}$, the
  instant of the floor. (a)~Unfiltered modulus $|u|$; (b)~filtered modulus
  $|\bar u|$, $\bar u = (1-\alpha^{2}\Delta)^{-1}u$ with $\alpha = 0.3$. The
  two panels share one colour scale, topped at the filtered peak
  $|\bar u|_{\max} = 39.2$; values above it saturate to white. The unfiltered
  peak at this instant is $170.6$, $4.4$ times the top of the scale, and the
  ratio of the two peaks is $0.23$. The black bar has length $\alpha$.
  (c)~Radial cut at the same instant. On $z=0$ the ansatz is exactly radial,
  $|u| = \tau^{-A}F\!\big(r/\sqrt{2\nu\tau}\big)\,W(r)$, so this cut carries
  everything the panels show: the ring of $|u|$ peaks at
  $\ell_r = 0.78\,\alpha$ and the whole core lies inside one filter length,
  so the filter flattens it into a plateau of width $\sim\alpha$ and, from
  here on, stops responding to its growth. The dark dotted line marks
  $r = \alpha$; the light dotted line, the outer edge of the construction,
  $r = \sqrt{2\nu\tau X_b}$. At this instant the two nearly coincide
  ($0.300$ and $0.286$), so the steep outer drop --- present identically in
  the unfiltered curve --- is the construction's window and not the filter.
  Static filter on the exact ansatz; not a simulation of the $\alpha$-model
  (see Scope).}
  \label{fig:piso-composite}
\end{figure}

\begin{figure}[t]
  \centering
  \includegraphics[width=\linewidth]{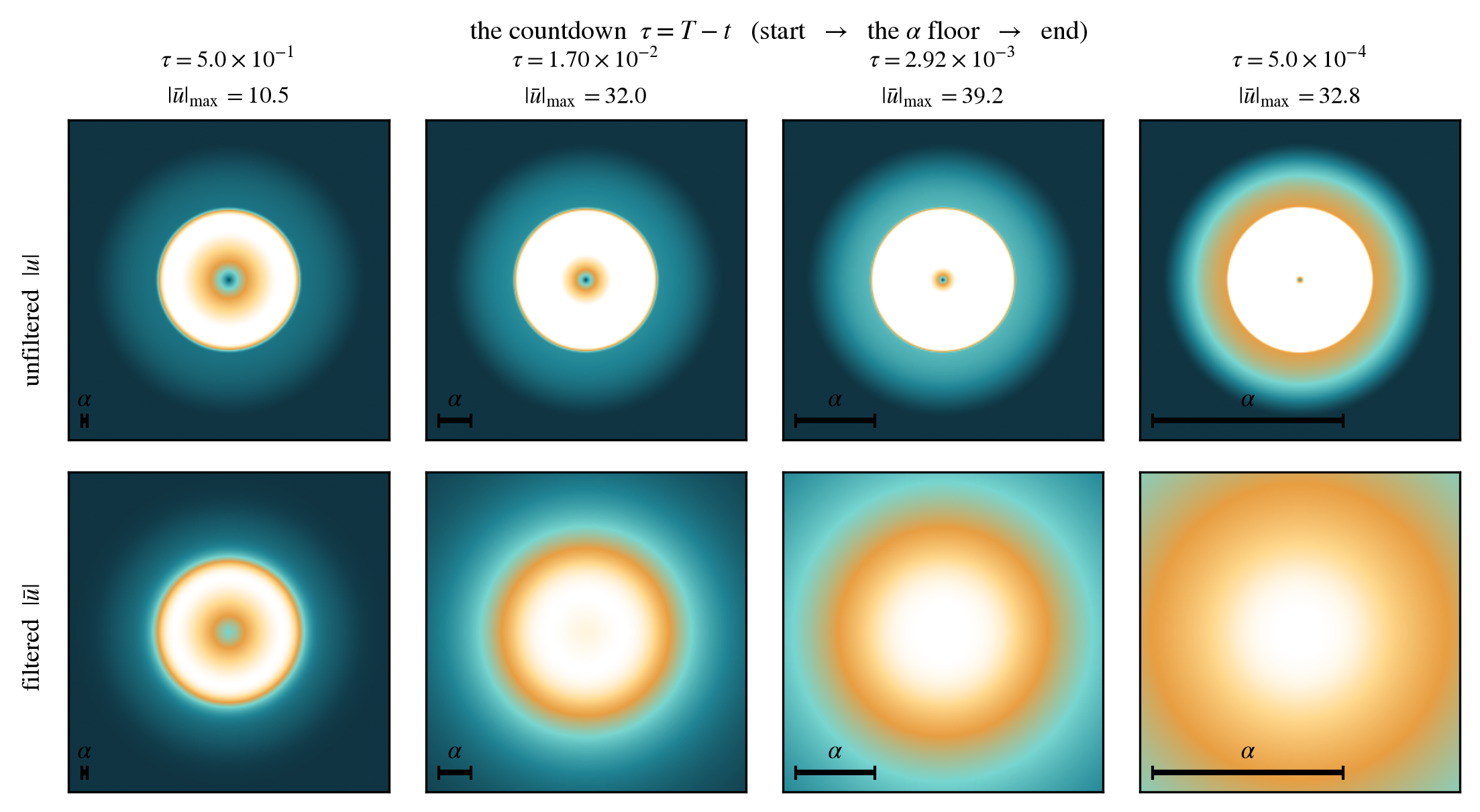}
  \caption{\textbf{The countdown $\tau = T - t$, unfiltered and filtered.}
  Four instants of the illustration of Figure~\ref{fig:piso-composite}: the
  start ($\tau = 5.0\times10^{-1}$), an intermediate frame
  ($1.70\times10^{-2}$), the floor ($2.92\times10^{-3}$) and the last frame
  ($5.0\times10^{-4}$). Top row, unfiltered $|u|$; bottom row, filtered
  $|\bar u|$ with $\alpha = 0.3$. Each column has its own colour scale, topped
  at that instant's filtered peak (printed above the column); white is
  saturation. The box follows the blob, $R = 1.15\sqrt{2\nu\tau X_{\max}}$,
  so the absolute scale is carried by the $\alpha$ bar: it grows from $2\%$
  of the width of the first frame to $60\%$ of the last, where it exceeds
  the half-width of the box ($\alpha = 0.3$ against $R = 0.252$) --- the
  whole last frame is $1.7\alpha$ across. The unfiltered peak grows
  $12.4 \to 69.4 \to 170.6 \to 419.5$; the filtered peak reaches its floor,
  $39.2$, and then \emph{decreases} to $32.8$ (Table~\ref{tab:piso-floor}).
  Same static-filter illustration as Figure~\ref{fig:piso-composite}.}
  \label{fig:piso-strip}
\end{figure}

\paragraph{What the figures show.}
Along the countdown the unfiltered peak diverges as $\tau^{-A}$, with
$A = \tfrac12 + h$ ($12.4 \to 419.5$ over the four instants of
Table~\ref{tab:piso-floor}), while the filtered peak rises to a finite
maximum, $39.2$ at $\tau_{\mathrm f} = 2.92\times10^{-3}$, and then turns
over, ending at $32.8$ while $|u|$ keeps growing; the ratio of the two peaks
falls from $0.85$ to $0.078$. The turning point is geometric. The core of the
ansatz has radial scale $\ell_r \propto \sqrt{\nu\tau}$, and the floor is
reached when it drops below the filter length ($\ell_r/\alpha = 0.78$ at
$\tau_{\mathrm f}$, $0.32$ at the end). Past that instant the filter no
longer resolves the ring: what it returns is an average of $|u|$ over a disc
of radius $\sim\alpha$, and the planar integral of the core,
$\propto \tau^{\,1-A} = \tau^{\,1/2-h}$, decreases towards the blow-up time. The filter therefore caps the filtered peak at a level set by
$\alpha$, and the dimensional form of the cap follows from the same two
scalings: the floor is reached when $\sqrt{\nu\tau}$ is a fixed fraction of
$\alpha$, and the amplitude there is $\tau_{\mathrm f}^{-A}$,
\begin{equation}
  \tau_{\mathrm f} = C\,\frac{\alpha^{2}}{\nu},
  \qquad
  \max|\bar u|_{\mathrm{floor}} \;\propto\; \tau_{\mathrm f}^{-A}
  \;\propto\; \alpha^{-(1+2h)} .
  \label{eq:piso-leyes}
\end{equation}
Both laws are measured in the same realization: over the five-point sweep
$\alpha \in \{0.15,\,0.212,\,0.3,\,0.424,\,0.6\}$ of
Table~\ref{tab:piso-floor} the $\log$--$\log$ slopes are $+2.006$ for
$\tau_{\mathrm f}$ against $\alpha$ and $-1.020$ for the floor level, against
the exponents $+2$ and $-(1{+}2h) = -1.02$ of Eq.~\eqref{eq:piso-leyes}.

\paragraph{What they do not show.}
These are not trajectories of the $\alpha$-model and say nothing about its
dynamics; they display the operator the model is built from, acting on the
exact blow-up solution. Two readings should be resisted. The outer decay in
panel~(c) of Figure~\ref{fig:piso-composite} (in page $10$) past $r \approx 0.29$, is the
edge of the construction --- the core ends at $X = X_b = 14$ and the field is
windowed smoothly to zero beyond it --- and not the filter's action.
And the fact that the floor time $\tau_{\mathrm f} = C\alpha^{2}/\nu$ and the
separant constant $\alpha^{2}\nu\,(m{+}4)!/m!$ are built from the same two
parameters is a coincidence of ingredients, not a theorem: the combinations
differ ($\alpha^{2}/\nu$ is the viscous diffusion time across one filter
length; $\alpha^{2}\nu$ is not a time), one belongs to the static filter and
the other to the differential ideal.

\begin{table}[t]
  \centering
  \small
  \begin{tabular}{lccccc}
  \toprule
  instant & $\tau$ & $\max|u|$ & $\max|\bar u|$ & $|\bar u|_{\max}/|u|_{\max}$ & $\ell_r/\alpha$ \\
  \midrule
  start        & $5.00\times10^{-1}$ &  12.38 & 10.50 & 0.848 & 10.21 \\
  intermediate & $1.70\times10^{-2}$ &  69.41 & 32.00 & 0.461 &  1.88 \\
  floor        & $2.92\times10^{-3}$ & 170.63 & 39.18 & 0.230 &  0.78 \\
  end          & $5.00\times10^{-4}$ & 419.47 & 32.76 & 0.078 &  0.32 \\
  \bottomrule
  \end{tabular}
  \caption{Readouts along the countdown for $h = 0.01$, $\nu = 1$,
  $\alpha = 0.3$: 48 geometrically spaced frames over
  $\tau \in [5\times10^{-1},\,5\times10^{-4}]$ on a $256\times256$ grid.
  $\ell_r = \sqrt{2\nu\tau X_{*}}$ is the radius at which $|u|$ peaks, i.e.\
  the radial scale of the core. The scaling laws quoted in the text come
  from a five-point sweep $\alpha \in \{0.15,\,0.212,\,0.3,\,0.424,\,0.6\}$
  with a finer scan of 160 instants per $\alpha$, so that the position of
  the maximum is not quantised by the frame grid: $\log$--$\log$ slopes
  $+2.006$ ($\tau_{\mathrm f}$ vs $\alpha$) and $-1.020$ (floor level vs
  $\alpha$).}
  \label{tab:piso-floor}
\end{table}

\paragraph{Method.}
The field is the exact construction of~\cite{openai2026navierstokes}
(eqs.~(4.3)/(4.5), (4.7), (4.29)), evaluated on the mid-plane $z=0$, where it
is exactly radial. The filter is the continuum Helmholtz operator
$(1-\alpha^{2}\Delta)^{-1}$ realized in free space, the natural domain of the
ansatz, by spectral evaluation on the figure's box zero-padded by $6\alpha$
on each side (periodic wrap-around error $\sim 2K_0(6) \approx 2.5\times10^{-3}$);
the panels are then cropped back to $[-R,R]^{2}$. The padding is essential
once $R \lesssim \alpha$: filtering the un-padded periodic box degenerates
into averaging over the box, and the floor disappears. The interactive demo
of the project realizes the same operator with finite differences in $y$ and
Dirichlet walls ($\bar u = 0$ at $y = \pm R$); the two realizations agree in
the core to $1.3\%$ at the floor and differ only near the walls (by up to
$36\%$ of the peak, on the outermost rows) at the last instants, when the
filtered halo reaches them. All quoted readouts are interior quantities. The
pulse measurements of Section~\ref{subsec:escudo} use the Dirichlet
realization on boxes aligned to the pulse, for which the prescribed pulse is
an exact mode of the operator.


\subsection*{The spectral shield in wavelength: the half point lies at
\texorpdfstring{$2\pi\alpha$}{2 pi alpha}}
\label{subsec:escudo}

The preceding figures measure the floor that the Helmholtz filter imposes on
the \emph{amplitude} of the blow-up. Theorem~1 (part~V) states the shield in
\emph{scale}: an oscillatory pulse that tries to shrink its period below the
filter scale is choked. The measurement below locates that boundary and puts
a number on the word ``choked''. The protection is not a threshold but a
smooth roll-off: a pulse of wavelength $\lambda$ is transmitted with the
factor
\begin{equation}
  \mathcal{T}(\lambda) \;=\; \frac{1}{1 + (2\pi\alpha/\lambda)^{2} + \alpha^{2}k_y^{2}},
  \qquad k_y = \frac{\pi}{2R},
  \label{eq:escudo-transfer}
\end{equation}
where the last term comes from the transverse envelope of the prescribed
pulse (it vanishes at the walls $y = \pm R$ of its measurement box) and is
negligible when the box is wide, $R \gg \alpha$. The half point,
$\mathcal{T} = 1/2$, lies at
\begin{equation}
  \lambda_{1/2} \;=\; 2\pi\alpha \;\approx\; 6.28\,\alpha ,
  \label{eq:escudo-medio}
\end{equation}
that is, at wavenumber $k = 1/\alpha$, where $\alpha^{2}k^{2} = 1$; at
$\lambda = \alpha$ the transmission is already down to
$\mathcal{T} = 1/(1+4\pi^{2}) = 2.5\%$. The two natural phrasings of the
shield are therefore not the same boundary: ``choked above $k = 1/\alpha$''
names the half point, $\lambda = 2\pi\alpha$, whereas ``choked below
$\lambda = \alpha$'' names a scale at which the pulse is already suppressed
forty-fold. The characteristic scale of the shield is $\alpha$ in
\emph{wavenumber} and $2\pi\alpha$ in \emph{wavelength}. In wavenumber the
measurement meets the linearised normal-mode analysis of the filtered
boundary-layer equations, whose shield family of wall-normal exponents is
anchored at $\pm1/\alpha$ independently of the streamwise wavenumber: the
same scale, here measured on the filter itself acting on the exact field
rather than read off the linearised symbol.

\begin{figure}[tbp]
  \small
  \centering
  \includegraphics[width=0.75\linewidth]{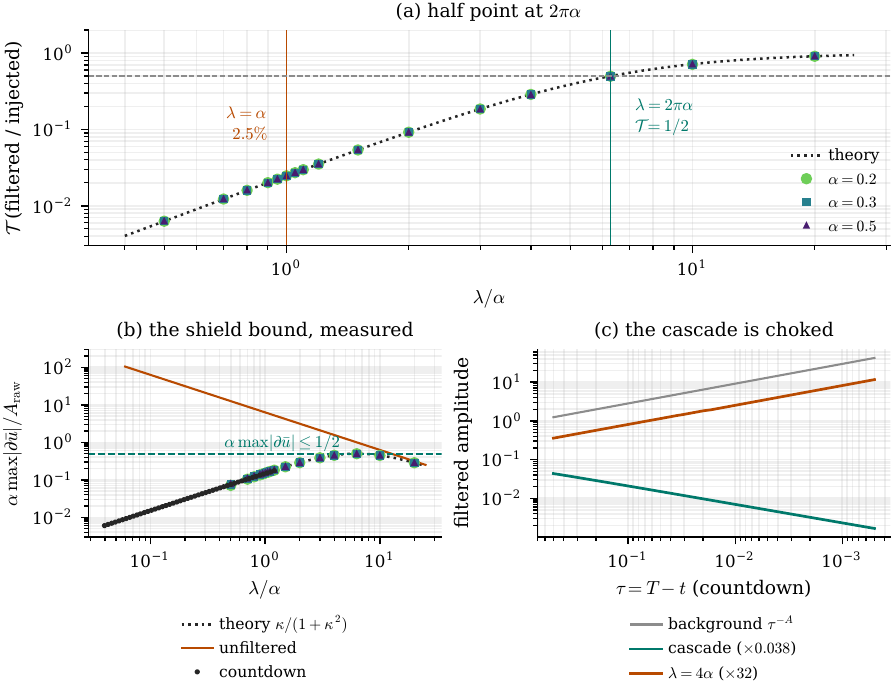}
  \caption{\textbf{The spectral shield of the Helmholtz filter, measured.}
  A prescribed oscillatory pulse of wavelength $\lambda$ and amplitude
  $A_{\mathrm{raw}} = \varepsilon\max|u|$, $\varepsilon = 0.1$, is superposed on the exact
  blow-up ansatz of~\cite{openai2026navierstokes} and passed through the
  Helmholtz filter $\bar u = (1-\alpha^{2}\Delta)^{-1}u$ of the
  Leray-$\alpha$ family~\cite{cheskidov2005}. The filter is linear, so the
  transmission of the pulse does not depend on the background; amplitudes are
  read by modal projection on boxes aligned to the pulse (at least two full
  wavelengths), so the measurement carries no spectral leakage.
  \textbf{(a)}~Transmission $\mathcal{T}$ (filtered over injected amplitude)
  against $\lambda/\alpha$ for $\alpha = 0.2,\,0.3,\,0.5$ (markers; the three
  collapse onto one curve) and Eq.~\eqref{eq:escudo-transfer} in the wide-box
  limit $k_y \to 0$ (dotted). The vertical lines mark $\lambda = \alpha$,
  where $\mathcal{T} = 2.5\%$, and $\lambda = 2\pi\alpha$, where
  $\mathcal{T} = 1/2$. Measured and predicted values agree to four significant
  figures from $\lambda = 0.5\alpha$ to $20\alpha$, as they must for an exact
  mode of the realized operator.
  \textbf{(b)}~The same data as a bound on the filtered gradient, normalised
  by the \emph{injected} amplitude: the points fall on $\kappa/(1+\kappa^{2})$,
  $\kappa = 2\pi\alpha/\lambda$, and never exceed the ceiling
  $\alpha\max|\partial\bar u|/A_{\mathrm{raw}} \le 1/2$, attained at
  $\lambda = 2\pi\alpha$,
  while the unfiltered gradient, $\kappa$, diverges as the pulse shortens.
  The small dots are the cascading pulse of panel~(c).
  \textbf{(c)}~The countdown $\tau = T - t$ of Figure~\ref{fig:piso-strip},
  $\alpha = 0.3$, with two pulses riding the blow-up: one that shortens
  geometrically, $\lambda_k = 1.2\alpha\,(0.93)^{k}$ over the 48 frames, and
  one held at $\lambda = 4\alpha$. The first ends at $0.038$ times its initial
  filtered amplitude although its injected amplitude grows $34$-fold; the
  second tracks the background $\tau^{-A}$ ($\times 32$ against $\times 34$).
  Static filter with a prescribed pulse on the exact ansatz; not a simulation
  of the $\alpha$-model dynamics.}
  \label{fig:escudo-rolloff}
\end{figure}

Panel~(c) of Figure~\ref{fig:escudo-rolloff} places the roll-off on the
countdown. The cascading pulse shortens by the factor $0.93$ per frame, which
is the contraction $\sqrt{\tau_{k+1}/\tau_k}$ of the blob's own box between
consecutive frames: it stays self-similar to the core, $\ell \propto
\sqrt{\tau}$, which is the structure the cascade of Theorem~1 (part~I) has to
sustain. Its injected amplitude, $\varepsilon\max|u|$, grows $34$-fold
over the 48 frames; its filtered amplitude ends at $0.038$ of its initial
value. The pulse held at $\lambda = 4\alpha$, on the other hand, is
transmitted at $\mathcal{T} \approx 0.29$ throughout and simply rides the
background ($\times 32$ against the background's $\times 34$; the small
deficit is the envelope term $\alpha^{2}k_y^{2}$, which grows as the
measurement box shrinks with the blob). The filter chokes the attempt to
shrink the scale and leaves the large-scale motion alone.

Two remarks on reading Eq.~\eqref{eq:escudo-transfer}. First, the cutoff is
soft. Below the half point the transmission decays quadratically,
$\mathcal{T} \simeq (\lambda/2\pi\alpha)^{2}$: it is $9.2\%$ at
$\lambda = 2\alpha$, $3.5\%$ at $1.2\alpha$, $2.5\%$ at $\alpha$ and $1\%$ at
$0.63\alpha$, with no scale at which the pulse is switched off. This is why
the qualitative statement of Theorem~1 --- below the filter scale a pulse is
choked --- holds, while the sharper reading of a cutoff \emph{at}
$\lambda = \alpha$ does not: nothing happens at $\alpha$ that does not also
happen, more gently, at $2\alpha$ and, more strongly, at $\alpha/2$. Second,
the same expression bounds the filtered derivatives. Normalising by the
injected rather than the filtered amplitude turns the roll-off into the
ceiling of panel~(b), $\alpha\max|\partial\bar u| \le A_{\mathrm{raw}}/2$, with
equality at
the half point $\lambda = 2\pi\alpha$; one order higher,
$\alpha^{2}\max|\partial^{2}\bar u| \le A_{\mathrm{raw}}$ for every $\lambda$, since
the
corresponding factor is $\kappa^{2}/(1+\kappa^{2}) < 1$. This second
bound is the wavelength-space counterpart both of the amplitude floor of the
preceding figures and of the $1/\alpha^{2}$ that caps the coefficients
$a_{m+4}$ in the recurrence of Theorem~1: whatever the wavelength of the
pulse, the curvature of the filtered field never exceeds the injected
amplitude divided by $\alpha^{2}$.

\newpage 

\section{Conclusions}
In this work, we have established a rigorous differential algebraic framework 
to study instabilities inherent in classical boundary layer formulations. 
By mapping the saturation monoids of the Leray-$\alpha$ boundary layers, 
we showed that applying the Helmholtz operator performs an \emph{Algebraic Surgery on the Boundary Layers}, fundamentally altering the system's differential structure. 
We proved that this differential filter modifies the system's separant from a vanishing dynamic variable into a strictly positive structural constant ($\alpha^2 \nu$). 
Because this regularized saturation monoid, $M_\alpha = \langle \alpha^2 \nu \rangle^\infty$, can never evaluate to zero, the singular branch of the ideal is mathematically forced to be the empty set ($\emptyset$). This guarantees the extraction of a globally regular Coherent Autoreduced Set via the Rosenfeld-Gröbner algorithm.

Physically, this constant separant acts as a shield. 
By mathematically capping all high-order spatial derivatives at $1/\alpha^2$, 
the system progressively constraints any oscillatory pulse attempting to shrink its spatial period below the $\alpha$-scale. Therefore, the high-frequency cascades required to drive finite-time velocity blowups such as those recently constructed for the Navier-Stokes equations are algebraically blocked. We formally prove, by reduction to the absurd, that such paradoxes cannot manifest in the regularized $\alpha$-model because the accumulation of rotational energy is topologically prevented.

  The accompanying measurements point the same way: on the exact
  self-similar blow-up, the filter imposes a finite floor on the filtered
  peak, and its transmission rolls off smoothly with half point at
  $\lambda = 2\pi\alpha$ --- the shield is a scale, not a threshold ---
  capping the filtered gradient at half the injected amplitude and choking a
  self-similarly contracting pulse ($\times 0.038$) while leaving a pulse
  held above the filter scale untouched. The singularity is excluded twice
  over: algebraically, because the separant never vanishes, and spectrally,
  because the regularized field cannot carry the cascade that would produce one.

Future investigations will focus on integrating these mathematically stable boundary coefficients into wall-modeled Large Eddy Simulations (LES)
of other boundary layers, providing a bridge between differential algebra and fluid dynamics.

\section*{Acknowledgments}
We gratefully acknowledge Universidad Iberoamericana Ciudad de México, and in
particular its Artificial Intelligence Laboratory, for access to the computing
infrastructure used for the experiments reported in this work. We also thank our
families, friends and colleagues for their constant support.

\bibliography{bibliography} 
\bibliographystyle{plainnat} %

\end{document}